\documentclass[12pt,reqno]{amsart}
\title[\protect\parbox{0.95\linewidth}{Optimal transport with optimal transport cost: the
Monge--Kantorovich problem on Wasserstein spaces}]{Optimal transport with optimal transport cost: the
Monge--Kantorovich problem on Wasserstein spaces}
\date{\today}
\author{Pedram Emami}
\email{emami1@ualberta.ca, pass@ualberta.ca}
\author{Brendan Pass}
\thanks{The work of PE was completed in partial fulfillment of the requirements for a doctoral degree in
mathematics at the University of Alberta. It is a pleasure to thank Lorenzo Dello Schiavo and an
anonymous referee for many insightful comments on an earlier version of this manuscript.}
\address{University of Alberta, Edmonton, Alberta, Canada}

\DeclareSymbolFont{let}{U}{zeur}{m}{n}
\SetSymbolFont{let}{bold}{U}{zeur}{b}{n}
\DeclareSymbolFont{symb}{U}{zeus}{m}{n}
\SetSymbolFont{symb}{bold}{U}{zeus}{b}{n}
\DeclareSymbolFont{lsymb}{U}{zeuex}{m}{n}
\DeclareMathSymbol{\psy}{\mathalpha}{let}{"20}

\usepackage{amsmath,amsthm,amssymb,mathrsfs,bm}
\usepackage{float,caption}
\usepackage{geometry,graphicx,tabularx,booktabs}
\usepackage[dvipsnames]{xcolor}
\usepackage{enumitem}
\usepackage{hyperref}

\definecolor{heavyred}{cmyk}{0,1,1,0.25}
\definecolor{heavyblue}{cmyk}{1,1,0,0.25}
\hypersetup{
  pdftitle=,
  pdfpagemode=UseOutlines,
  pageanchor=false,
  citebordercolor=0 0 1,
  colorlinks=true,
  allcolors=heavyred,
  breaklinks=true,
  pdfauthor=,
  pdfpagetransition=Dissolve,
}
\newtheoremstyle{thmStyle}{}{}{\it}{}{\bfseries}{.}{ }{}
\newtheoremstyle{corStyle}{}{}{\it}{}{\bfseries}{.}{ }{}
\newtheoremstyle{roStyle}{}{}{}{}{\it\bfseries}{.}{ }{}
\newtheoremstyle{prfStyle}{}{}{}{}{\it\bfseries}{.}{ }{}
\newtheoremstyle{exStyle}{}{}{}{}{\bfseries}{.}{ }{}  
{
  \theoremstyle{thmStyle}
  \newtheorem{theorem}{Theorem}
  \newtheorem{definition}[theorem]{Definition}
  
  \newtheorem{lemma}[theorem]{Lemma}
   
  \newtheorem{assumption}[theorem]{Assumption}
  \newtheorem{remark}[theorem]{Remark}

}
{
    \theoremstyle{prfStyle}
  \newtheorem*{prf}{Proof}
}
{
    \theoremstyle{exStyle}
  
}

\makeatletter
\let\cal\mathcal%

\def\R{\mathbb{R}}%
\def\bD{\textbf{\text{D}}}%
\def\compose{\circ}%
\def\ddt{\dfrac{d}{dt}}%
\def\(#1\){\left(#1\right)}%
\def\<#1\>{\langle#1\rangle}%
\def\]#1\]{\left\{#1\right\}}%
\def\|#1|{\@ifnextchar|{\left\lVert#1\right\rVert\sbox0}{\left|#1\right|}}%
\def\norm{\left\lVert\,\cdot\,\right\rVert}%
\def\half{\frac{1}{2}}%
\def\endofprf{\hbox{}\hfill$\square$}%
\def\obar#1{\overline{\sbox1{$#1$}\dp1=0pt\box1}}%
\def\volg{\mathrm{dvol_g}}%
\def\smoothFP{\cal{FC}^\infty}%
\def\pM{\cal P(M)}%
\def\pac{\cal P^\mathrm{ac}(M)}%
\def\PM{\cal{P}(M)}%
\def\Exp{\mathrm{exp}}%
\def\RP{\mathbb{P}}%
\def\Xz{\mathfrak{X}^0}%
\def\Xmu{{\mathfrak{X}_\mu}}%
\def\Xinf{\mathfrak{X}^\infty}%
\def\flow{\psy^{w,t}}%
\makeatother

\begin{document}
\maketitle
\bibliographystyle{amsplain}

\begin{abstract}
We consider the Monge--Kantorovich problem between two random measures. More precisely, given probability
measures $\mathbb{P}_1,\mathbb{P}_2\in\mathcal{P}(\mathcal{P}(M))$ on the space $\mathcal{P}(M)$ of
probability measures on a smooth compact manifold, we study the optimal transport problem between
$\mathbb{P}_1$ and $\mathbb{P}_2 $ where the cost function is given by the squared Wasserstein distance
$W_2^2(\mu,\nu)$ between $\mu,\nu \in \mathcal{P}(M)$. Under appropriate assumptions on $\mathbb{P}_1$,
we prove that there exists a unique optimal plan and that it takes the form of an optimal map. An
extension of this result to cost functions of the form $h(W_2(\mu,\nu))$, for strictly convex and
strictly increasing functions $h$, is also established. The proofs rely heavily on a recent result of
Dello Schiavo \cite{schiavo2020rademacher}, which establishes a version of Rademacher's theorem on
Wasserstein spaces.
\end{abstract}

\section{Introduction} This paper focuses on optimal transport when the ambient space itself is the
infinite-dimensional space of probability measures on a Riemannian manifold, endowed with the Wasserstein
metric.

Assume that $X$ and $Y$ are metric spaces. Given probability measures $\mu\in\cal P(X)$ and $\nu\in\cal
P(Y)$, and a cost function $c:X\times Y\rightarrow\R$, a fairly general version of Monge's original
formulation of the optimal transport problem is as follows \cite{monge1781}:

\begin{equation}\label{eqn: Monge problem}
\inf_{\stackrel{T:X\rightarrow Y,}{T_\#\mu =\nu}}\,\int_{X}c(x,T(x))\,d\mu,
\end{equation}
where the infimum is over measurable maps $T:X \rightarrow Y$ pushing $\mu$ forward to $\nu,$
$T_{\#}\mu=\nu$; that is, $\mu(T^{-1}(A)) =\nu(A)$ for all Borel $A \subseteq Y$.

Non-linearity and a lack of compactness make this problem challenging on its own.  Kantorovich's
relaxed version \cite{kantorovich1942} sidesteps these issues by seeking transport \emph{plans} rather
than maps.  Explicitly, his formulation of the problem is:

\begin{equation} \label{eqn: kantorovich problem}
\inf_{\gamma\in\Pi(\mu,\nu)}\int_{X\times Y}c(x,y)\,d\gamma,
\end{equation}
where $\Pi(\mu,\nu)$ is the set of all measures on $X\times Y$ such that $(\pi_X)_{\#}\gamma=\mu$,
and $(\pi_Y)_{\#}\gamma=\nu$, where $\pi_X$ and $\pi_Y$ are the canonical projections.  Kantorovich's
problem amounts to an infinite-dimensional linear program, and, in contrast to \eqref{eqn: Monge
problem}, it is straightforward to show that a solution $\gamma$, known as an \emph{optimal plan}, exists
under reasonable assumptions on $c$, $\mu$ and $\nu$.

Problems \eqref{eqn: Monge problem} and \eqref{eqn: kantorovich problem} have generated a lot of research
over the past several decades, and have many applications, reviewed in the monographs of Villani
\cite{villani2021topics, villani2009} and Santambrogio \cite{santambrogio2015optimal}.  In many of these
applications, the cost function $c(x,y) =d^2(x,y)$ is the metric distance squared on a common metric
space $X=Y$, and we focus on this case in what follows.    We highlight here two important themes, which
are pertinent to this paper. First, the problem induces an extended distance $W_2$ on the space
$\mathcal{P}(X)$ of probability measures on $X$.  This distance is known as the Wasserstein distance, and
$(\mathcal{P}(X), W_2 )$ as the Wasserstein space.    Second, for many choices of $X$, under additional
assumptions on $\mu$, it is possible  to prove that the solution to \eqref{eqn: kantorovich problem} is
unique, and concentrates on the graph of a function $T: X \rightarrow X$ pushing $\mu$ forward to $\nu$.
This $T$ in turns solves the Monge problem  \eqref{eqn: Monge problem}, and is known as an \emph{optimal
map}.  This result was first established on convex Euclidean domains by Brenier
\cite{brenier1991polar}, and has since been extended by many authors to a wide variety of other geometric
situations, notably including Riemannian manifolds by McCann \cite{mccann2001polar}, as well as the
Heisenberg group \cite{AmbrosioRigot04}, general sub-Riemannain manifolds \cite{FigalliRifford10},
Alexandrov spaces \cite{Bertrand08}, and several infinite-dimensional spaces which we will touch
on in more detail shortly.

Our purpose here is to establish a similar result when the underlying metric space $X$ is itself the
Wasserstein space over a compact Riemannian manifold\footnote{Throughout this paper, the term manifold will always mean a manifold without boundary.} $M$, $(X,d) = (\mathcal{P}(M), W_2 )$.  The optimal
transport problem on Wasserstein spaces is natural, as it can be thought of as  quantifying the
distance between two random (probability) measures.  In addition, we note that a substantial amount of
recent research on the applied side of optimal transport has focused on situations where the marginals
$\mu$ and $\nu$ in optimal transport problems on finite dimensional spaces $X$ and $Y$, or other measures
of interest, are not known with certainty; see for instance \cite{BlanchetMurthy2019, TrigilaTabak2016,
GhossoubSaundersZhang2024}. This uncertainty or ambiguity is  expressed in different ways depending on
the precise context, but it seems reasonable to conjecture that problems where marginal uncertainty is
reflected by a \emph{distribution} of possible marginals on Wasserstein space (instead of precisely known
prescribed marginals) may be relevant in some of these applications. This leads naturally to our problem.

Let us say a few words about the proof of our main result. For general choices of $X$, it is well
understood that one can only establish the existence of optimal maps under certain regularity assumptions
on $\mu$; namely, one needs to exclude the possibility of $\mu$ charging sets which are "small" in an
appropriate sense.  Although more refined notions exist (see \cite{gigli2011inverse} for a sharp
condition when $X \subseteq \mathbb{R}^n$), a straightforward sufficient hypothesis (a version of which
we will adopt here) is that $\mu$ is absolutely continuous with respect to an appropriate reference
measure.  From a technical perspective, this hypothesis arises in the proof  of the existence of optimal
maps by ensuring that we can differentiate Kantorovich potentials (defined in the next section) $\mu$
almost everywhere.  Since these potentials are well known to be locally Lipschitz, Rademacher's theorem
makes this work when the reference measure is, for example, the Lebesgue measure on
$\mathbb{R}^n$, or Riemannian volume on a Riemannian manifold. It is also worth mentioning the case of
the Wiener space, which, like our setting here, is infinite-dimensional, where Monge solutions
were established in \cite{FeyelUstunel04} and \cite{Kolesnikov2004}.  In that case, the natural reference
measure is the Wiener measure.   Existence and uniqueness of optimal maps were proven more
generally on separable Hilbert spaces in \cite{ambrosio2008gradient}, where all non-degenerate Gaussian
measures play the role of the reference measure in a certain sense, and on configuration spaces in \cite{decreusefond2008wasserstein}.

In our particular infinite-dimensional setting $X =\mathcal{P}(M)$, the choice of an appropriate
reference measure is non-trivial.  However, recent work of Dello Schiavo establishes a version of
Rademacher's theorem on Wasserstein spaces for a class of reference measures satisfying certain
assumptions \cite{schiavo2020rademacher}.  Here, we show that with this result in hand, the fairly
standard proof strategy for the existence of optimal maps on many metric spaces can indeed be adapted to
the Wasserstein space.

The paper is organized as follows.  In the following section, we review some basic facts and notation on
optimal transport and Riemannian geometry which we will use later on. In Section 3 we recall relevant
information on the space of probability measures on a Riemannian manifold, and we state and prove our
main theorem on the existence and uniqueness of optimal maps on this space in Section 4.  Section 5 is
reserved for an extension of this result to more general cost functions.

\section{Preliminaries}
\subsection{Basic tools for optimal transport}

A key tool in the analysis of optimal transport plans is the Kantorovich dual problem, which reads:
\begin{equation}\label{eqn: dual problem}
\sup_{\substack{\phi\in C(X),\psi\in C(Y)}}
\]\int_X{\phi(x)\,d\mu}+\int_Y{\psi(y)\,d\nu};\quad \phi(x)+\psi(y)\leq c(x,y)\].
\end{equation}
Solutions $\varphi$ and $\psi$ to this problem are called \textit{Kantorovich potentials}.  They can be
chosen to have a particular structure, expressed in the following definition.

\begin{definition}
A function $\phi:X \rightarrow \mathbb{R}$ is called \emph{$c$ -concave} if 
\begin{equation*}
\phi(x) =\psi^c(x):=\inf_{y \in Y}c(x,y) - \psi(y)
\end{equation*}
for some function $\psi:Y \rightarrow \mathbb{R}$. Similarly, a function $\psi:Y \rightarrow \mathbb{R}$
is called $c$-concave if
\begin{equation*}
\psi(y) =\phi^c(y):=\inf_{x \in X}c(x,y) - \phi(x)
\end{equation*}
for some $\phi:X \rightarrow \mathbb{R}$.
\end{definition}

The following well known theorem, asserting duality between \eqref{eqn: kantorovich problem} and
\eqref{eqn: dual problem}, makes precise the remarks above about the connection between \eqref{eqn:
kantorovich problem}, \eqref{eqn: dual problem} and the notion of $c$-concavity, and will play a crucial
role in this paper.

\begin{theorem}\label{Kantorovich-duality-extension}
Let $\mu,\nu$ be Borel probability measures on compact metric spaces, $X$ and $Y$, respectively, and $c:X
\times Y \rightarrow \mathbb{R}$ a continuous cost function.

Then the infimal value in \eqref{eqn: kantorovich problem} equals the supremal value in \eqref{eqn: dual
problem}, and there exists a minimizer $\gamma$ in \eqref{eqn: kantorovich problem} and a maximizer
$(\phi,\psi) =(\psi^c,\phi^c)$ in \eqref{eqn: dual problem} such that both $\phi$ and $\psi$ are
$c$-concave.

Furthermore, for all solutions $\gamma$ to \eqref{eqn: kantorovich problem} and $(\phi,\psi)$ in
\eqref{eqn: dual problem} we have
\begin{equation*}
\phi(x) +\psi(y) =c(x,y)\qquad\gamma\text{ almost everywhere}.   
\end{equation*}
\end{theorem}

\subsection{Riemannian manifolds}
In this section, we introduce some notation and recall some known facts on Riemannian manifolds which we
will need later on.  Let $(M,g)$ be a compact connected Riemannian manifold with Riemannian distance
$d(x,y):=d_g(x,y)$. For the rest of this work, we will quite closely follow notation in
\cite{schiavo2020rademacher}.  We denote by $\volg$ the canonical volume measure on $M$ and the set of
all smooth functions on $M$ by $C^\infty(M)$. Moreover, we denote by $\Xz$ the space of all continuous
vector fields on $M$; that is, sections of the tangent bundle $TM$, endowed with the supremum norm
\begin{equation*}
\|w||_{\Xz}:=\sup_M\|w(x)|_g.
\end{equation*}
We denote by $\Xinf$ the algebra of smooth vector fields on $M$. Since $M$ is compact, every smooth
vector field $w\in\Xinf$ generates a unique global flow $\(\flow\)_{t\in\R}$; that is, each $\flow:M\to
M$ is a smooth orientation-preserving diffeomorphism such that

\begin{equation}\label{flow-of-smooth-vector-fields}
\dot{\psy}^{w,t}(x)=w(\flow(x))\quad\&\quad\psy^{w,0}(x)=x,\qquad x\in M,
\end{equation}
where $\dot{\psy}^{w,t}(x)$ is the velocity of the curve $t\mapsto\flow(x)$.

By Theorem \ref{Kantorovich-duality-extension}, the Kantorovich problem on $M$ with the cost function
$c(x,y)=d(x,y)^2$ always admits a solution. Moreover, the following famous theorem of McCann
\cite{mccann2001polar} implies the existence and uniqueness of an optimal  map whenever the source measure
lies in the set $\pac$ of probability measures which are absolutely continuous with respect to $\volg$.

\begin{theorem}[McCann 2001] \label{thm: McCann Monge solutions on manifolds} Let $\mu \in \pac$ and $\nu
\in \pM$.  Then the solution to \eqref{eqn: kantorovich problem} on $X=Y=M$ with $c(x,y) =d^2(x,y)$ is
unique and concentrated on the graph of a function $T:M \rightarrow M$.  This $T$ is in turn the unique
solution to \eqref{eqn: Monge problem}, and takes the form $T(x) =\exp_x(-\nabla \phi(x))$ for a
$\frac{d^2}{2}$ -concave function $\phi: M \rightarrow \mathbb{R}$.
\end{theorem}

\section{The space of Borel probability measures on a Riemannian manifold}
We now turn our attention to the space $\pM$ of probability measures on a Riemannian manifold $M$.
\subsection{Topological properties of \texorpdfstring{$\pM$}{P(M)}} Since $(M,g)$ is a compact metric
space, the space of all finite signed Borel measures on $M$, $\cal M(M)$, endowed with the total
variation norm, can be identified with the dual of $C(M)=C_b(M)$, thanks to the
Riesz--Markov--Kakutani theorem. Furthermore, $\pM$ is a weak$^*$ closed subset of the unit
sphere of $\cal M(M)$ and so by the Banach--Alaoglu theorem is weak$^*$ compact. Every $f\in
C_b(M)$ induces a bounded linear functional
\begin{equation*}
f^*:\PM\rightarrow\R,\qquad f^*\mu=\int_M f\,d\mu,
\end{equation*}
on $\pM$ which is called a \emph{potential energy}. When $f\in C^\infty(M)$ the corresponding $f^*$ is
called a \emph{smooth potential energy}. Motivated by studies on configuration spaces
\cite{decreusefond2008wasserstein}, there is
a class of smooth functions on $\pM$ which is relevant to the rest of our analysis.
\begin{definition}
A function $u:\PM\rightarrow\R$ is called a ``cylinder'' function if there exists a positive integer $k>
0$, a smooth function $F$ on $\R^k$ and $k$ smooth functions $f_1,f_2,\cdots,f_k\in C^\infty(M)$ such
that
\begin{equation*}
u(\mu)=F(f_1^*\mu,f_2^*\mu,\cdots,f_k^*\mu).
\end{equation*}
We denote the set of all cylinder functions on $\PM$ by $\smoothFP$.
\end{definition}
The fact that the Kantorovich problem on $M$ always admits a solution implies that for every two measures
$\mu,\nu\in\pM$, the cost of the transportation is a finite value. In fact, this value induces a metric
\begin{equation*}
W_2(\mu,\nu):=\(\int_{X\times Y}d(x,y)^2\,d\gamma(x,y)\)^\half,
\end{equation*}
where $\gamma$ is optimal in \eqref{eqn: kantorovich problem}, on the space $\pM$, known as the
\emph{$L^2$-Wasserstein metric}. The $L^2$-Wasserstein metric metrizes the weak$^*$ topology on $\pM$
\cite[Theorem 7.12]{villani2021topics} and turns it into a compact metric space itself. This space $(\pM,
W_2)$ is called the $L^2$-Wasserstein space over $M$.

\subsection{The Geometry of \texorpdfstring{$\PM$}{P(M)}}
We now turn to some geometric properties induced on $\pM$ by the metric $W_2$.

The Wasserstein space is often viewed as an infinite-dimensional manifold. The geometry and smooth
structure of this space have been well studied from different viewpoints by many researchers
\cite{ambrosio2008gradient, lott2008some, gigli2011inverse, gigli2009second, gangbo2011differential}. It
is well-known that $\PM$ is a geodesic space and one can identify the tangent space of $\PM$ at a point
$\mu$, denoted by $T_{\mu}^{\mathrm{Der}}\PM$, with the space $\Xmu:=\obar{\Xinf}^{L^2(M,\mu)}$, that is,
the abstract linear completion of $\Xinf$ with respect to the norm $\norm_{\Xmu}$ induced by the Hilbert
inner product
\begin{equation*}
\<w_0,w_1\>_\Xmu:=\int_M\<w_0(x),w_1(x)\>_g\,d\mu(x),
\end{equation*}
for every $w_0,w_1\in\Xinf$. Under this identification, every smooth vector field $w\in\Xinf$ induces a
canonical directional derivative for every cylinder function $u\in\smoothFP$. Recalling the unique global
flow
\eqref{flow-of-smooth-vector-fields} generated by $w$,  observe that $\psy^{w,t}$ induces a map on
$\PM$ via push-forward, that is,
\begin{equation*}
\Psi^{w,t}:\PM\rightarrow\PM,\qquad\Psi^{w,t}:=\psy^{w,t}_\#.
\end{equation*}
Then a straightforward calculation \cite[Lemma 6.2]{schiavo2020rademacher} shows that for every
$u\in\smoothFP$
\begin{equation*}
(\nabla_w u)(\mu):=\ddt{\Big|_{t=0}}(u\compose\Psi^{w,t})(\mu)= {\<\nabla u(\mu),w\>}_{\Xmu},
\end{equation*}
where $\nabla u$ is defined via\footnote{For justification on this definition, we refer the reader to
Dello Schiavo \cite[Page 5]{schiavo2020rademacher}.}
\begin{equation*}
\nabla u(\mu)(x):=\sum_i^k(\partial_iF)(f_1^*\mu,f_2^*\mu,\cdots,f_k^*\mu)\nabla f_i(x).
\end{equation*}
We next establish a simple lemma which we will require later.
\begin{lemma}\label{forward-compatibility} Let $\mu\in\pac$ and $w$ be a smooth vector field on $M$,
that is $w\in\Xinf$. Then $\mu_t:=\Psi^{w,t}(\mu)\in\pac$ for every $t\in\R$. Moreover, $\mu_t$ satisfies
the continuity equation
\begin{equation}\label{continuity-eq}
\dfrac{\partial\mu_t}{\partial t}+\nabla\cdot(\xi_t\mu_t)=0,
\end{equation}
where $\xi_t(x)=w(\psy^{w,t}(x))$.
\end{lemma}

\begin{prf}
Since $M$ is compact, by Eq.~(\ref{flow-of-smooth-vector-fields}), $\psy^{w,t}$ is a diffeomorphism and
therefore $\mu_t=\Psi^{w,t}(\mu)$ is absolutely continuous with respect to $\volg$ for every $t\in\R$.
For the second part of the lemma, one should observe that, since $\psy^{w,t}$ is a diffeomorphism,
Theorem 5.34 \cite{TrigilaTabak2016} implies that $\mu_t$ satisfies the continuity equation
(\ref{continuity-eq}).\\
\endofprf%
\end{prf}

We next state two results on derivatives along curves in the Wasserstein space.  The first one regards
differentiation of the squared Wasserstein distance:
\begin{theorem}[Derivative of the Wasserstein distance]\label{Wasserstein-derivative} 
Let $M$ be a Riemannian manifold. Let $\mu\in\pac$ and $\nu \in \cal P(M)$.  For a smooth vector field $w
\in\Xinf$, set $\mu_t = \Psi^{w,t}(\mu)$.  Then
\begin{equation}\label{Wasserstein-der}
\dfrac{d}{dt}\Big|_{t=0} W_2^2(\mu_t,\nu) = 2\int_M\<\nabla\phi(x),w(x)\>_g\,d\mu(x),
\end{equation}
where $\phi$ is a $(d^2/2)$-concave function such that 
\begin{equation*}
\nu=\Exp(-\nabla\phi)_\#\mu.
\end{equation*}
\end{theorem}

Differentiability of the squared Wasserstein distance is certainly well known; see, for example,
\cite[Theorem 8.13]{villani2021topics}, \cite[Theorem 23.9]{villani2009}, \cite[Corollary
10.2.7]{ambrosio2008gradient}. We were unable to find an exact reference for the version we use here, and
so provide a brief proof.

\begin{proof}
The result follows directly from Lemma 4.3 in \cite{schiavo2020rademacher} when $\nu \neq \mu$.  If $\nu
=\mu$, then $\nu$ is absolutely continuous with respect to volume, and so the result follows from Step 2
in the proof of Theorem 23.9 in \cite{villani2009} (noting by Lemma \ref{forward-compatibility} that
$\mu_t$ satisfies the continuity equation).
\end{proof}

The second differentiation result we will need applies to Lipschitz functions on $\PM$ and is a recent,
deep result of Dello Schiavo \cite{schiavo2020rademacher}. It amounts to a type of Rademacher theorem on
$\PM$ and its statement requires some additional assumptions. In particular, a statement of a Rademacher
type result requires a reference measure on the base space (Wasserstein space in our case) satisfying
certain properties. The assumption we will need is as follows.
 
\begin{assumption}\label{assumption: reference measures}
Let $\RP_0 \in \mathcal{P}(\PM)$ be a probability measure on the Wasserstein space satisfying:
\begin{itemize}
\item $\RP_0$ has no atoms.
\item $\RP_0$ verifies the following integration-by-parts formula. If $u,v\in\smoothFP$ and $w\in\Xinf$,
then there exists a measurable function $\mu\mapsto\nabla^*_w v\in\Xmu$ such that
\begin{equation*}
\int_{\PM} \nabla_w u(\mu)\cdot v\,d\RP_0(\mu)= \int_{\PM} u\cdot\nabla^*_w v(\mu)\,d\RP_0(\mu).
\end{equation*}
\item $\RP_0$ is quasi-invariant with respect to the action of the flows generated by $\Xinf$; that is, for all $w\in\Xinf$, $\RP_0$ and $\Psi^{w,t}_{\#}\RP_0$ are mutually absolutely continuous.  Furthermore, the Radon-Nikodym derivative
\begin{equation*}
R^w_r:=\frac{d\big(\Psi^{w,r}_{\#}\RP_0\big) \otimes dr}{d\RP_0 \otimes dr}, \text{  }r \in \mathbb{R}
\end{equation*}
satisfies, for $\RP_0$ a.e $\mu$,
\begin{equation*}
\mathcal{L}^1\text{-essinf}_{r \in (s, t)}R^w_r(\mu) >0
\end{equation*}
for all $s,t \in \mathbb{R}$ with $s \leq t$.
\end{itemize}
\end{assumption}
\begin{remark}
Although the preceding assumption appears fairly abstract, it can be verified on many concrete examples;
see \cite{schiavo2020rademacher}.
\end{remark}

The statement of the main Theorem from \cite{schiavo2020rademacher} will require the following concept.
\begin{definition}
We say that a probabability measure $\RP_0 \in \mathcal{P}(\PM)$ satisfies the Rademacher property if for
every $W_2$-Lipschitz continuous functions $U:\PM\rightarrow\R$ there exists a measurable section
\emph{$\bD U$} of $T^{\mathrm{Der}}\PM$, satisfying
\begin{equation}\label{eqn: L^2 limit}
\lim_{t\to 0}\dfrac{U(\Psi^{w,t}(\cdot)) - U(\cdot)}{t} = \<\emph{\bD}
U(\cdot),w\>_{\mathfrak{X}_{(\cdot)}}\quad \text{in}\quad L^2\(\PM,\RP_0\),
\end{equation}
for every $w\in\Xinf$.
\end{definition}

The full version of the following Theorem is found in \cite[Theorem 2.10]{schiavo2020rademacher}; in the
interest of brevity, and in an attempt to keep our arguments as transparent as possible, we include here
only the conclusion we will explicitly need.

\begin{theorem}[Dello Schiavo 2020]\label{Schiavo-der}
Let $\RP_0\in\mathcal{P}(\PM)$ satisfy Assumption \ref{assumption: reference measures}. Then it satisfies
the Rademacher property.
\end{theorem}
\begin{lemma}\label{cor: pointwise differentiability}
Assume that $\mathbb{P}_0$ satisfies the Rademacher property. Let $U:\PM\rightarrow\R$ be $W_2$-Lipschitz
and $ \{w_i\}_{i=1}^\infty$ a countable set of smooth vector fields on $M$. Then for $\mathbb{P}_0$
almost every $\mu$, we have
\begin{equation}\label{eqn: pointwise limit}
\lim_{t\to 0}\dfrac{U(\Psi^{w_i,t}(\mu)) - U(\mu)}{t} = \<\emph{\bD} U(\mu),w_i\>_\Xmu,\quad
\end{equation}
for all $i\in\mathbb{N}$.
\end{lemma}
\begin{proof}
The Rademacher property implies that for each $w_i$, \eqref{eqn: pointwise limit} holds for
$\mathbb{P}_0$ almost every $\mu$. That is,
\begin{equation*}
\mathbb{P}_0(\{\mu: \eqref{eqn: pointwise limit} \text{ fails for } w_i \}) =0.
\end{equation*}
Therefore,     
\begin{equation*}
\mathbb{P}_0(\{\mu: \eqref{eqn: pointwise limit}\text{ fails for some } w_i\})=\mathbb{P}_0(
\cup_{i=1}^{\infty}\{\mu: \eqref{eqn: pointwise limit} \text{ fails for } w_i \})=0,
\end{equation*}
as desired.
\end{proof}

The following is a special case of a well known result in optimal transport theory, and, combined with
the above result, will imply differentiability of Kantorovich potentials (that is, solutions to
\eqref{eqn: dual problem}) on $\PM$ in a certain sense.  The analogous result on $M$ (rather than $\PM$)
was originally established by McCann \cite{mccann2001polar}; the version below can be seen as an
immediate consequence of the Proposition in \cite[Box 1.8]{santambrogio2015optimal}.
\begin{lemma}\label{lem: Lipschitz potentials}
 Let $U:\PM \rightarrow \mathbb{R}$ be $W_2^2$-concave.  Then $U$ is $W_2$-Lipschitz.    
\end{lemma}

\section{The optimal transportation problem on the space \texorpdfstring{$\PM$}{PM}}
We are now ready to state and prove our main theorem, which concerns the optimal transport problem on
the Wasserstein space itself.  That is, in \eqref{eqn: kantorovich problem}, we take $X=Y=\PM$ for
a compact Riemannian manifold, and $c(\mu,\nu) = W_2^2(\mu,\nu)$.  Note that $X$ and $Y$ are then
compact.

To distinguish between measures \emph{in} $\PM$ (that is, measures on $M$) and measures \emph{on} $\PM$
(that is, elements of $\mathcal{P}(\PM)$), we will continue to use Greek letters such as $\mu$ and $\nu$
for the former and will use upper case blackboard letters such as $\mathbb{P}_1$ and $\mathbb{P}_2$ to
denote the latter.

In the Theorem below, we assume the Rademacher property on the reference measure $\RP_0$.  The work of
Dello Schiavo \cite{schiavo2020rademacher}, summarized in Theorem \ref{Schiavo-der} above, identifies
conditions under which the Rademacher property holds (Assumption \ref{assumption: reference measures}).
These conditions are not strictly necessary, however; there are examples of references measures violating
them for which the Rademacher property holds, including the Dirichlet-Ferguson measure
\cite{DelloSchiavo2022}\footnote{Note, however, that Theorem \ref{Monge-extension} does not apply if
$\RP_0$ is the Dirichlet-Ferguson measure, as it concentrates on purely atomic measures, so that the
assumption $\RP_1(\pac)=1$ is violated for any measure $\RP_1$ which is absolutely continuous with
respect to $\RP_0$.}.  We therefore state our result under the Radmacher property assumption rather than
Assumption \ref{assumption: reference measures}, making it potentially more widely applicable.
\begin{theorem}\label{Monge-extension}
Let $\RP_1,\RP_2$ be Borel probability measures on $X:=\PM$ and $Y:=\PM$, respectively,  $\PM$ be
equipped with a reference measure $\RP_0$ satisfying the Rademacher property and
$c(\mu,\nu)=W_2^2(\mu,\nu)$ for every $\mu,\nu\in\pM$. Assuming that $\RP_1$ is absolutely continuous
with respect to $\RP_0$ and $\RP_1(\pac)=1$, the optimal transportation problem \eqref{eqn: kantorovich
problem} between $(X,\RP_1)$ and $(Y,\RP_2)$ under this cost function admits a unique solution
$\mathbb{P}_{1,2}$ which is concentrated on the graph of a map $T:X\to Y$; that is, $T(\mu)=\nu$ for
$\RP_{1,2}$ $\mathrm{ a.e.}\,(\mu,\nu)$ and $T_\#\RP_1=\RP_2$. Furthermore, $T$ is the unique solution to
\eqref{eqn: Monge problem}.
\end{theorem}
\begin{prf}
Let $\mathbb{P}_{1,2}\in\mathcal{P}(\PM \times \PM)$ solve the Kantorovich problem \eqref{eqn:
kantorovich problem} and $U$ and $V$ be $W_2^2$-concave functions solving the dual problem.  We then
have, by Theorem \ref{Kantorovich-duality-extension} for all $\mu,\nu\in\PM$,
\begin{equation}\label{eqn: dual inequality}
U(\mu) + V(\nu) \leq W_2^2(\mu,\nu),
\end{equation}
with equality $\mathbb{P}_{1,2}$ almost everywhere.  $U$ is Lipschitz by Lemma \ref{lem: Lipschitz
potentials}. Fix a countable $W=\{w_i\}_i \subseteq \Xinf$ which is dense in $\Xz$ with respect to the
sup norm. Then Corollary \ref{cor: pointwise differentiability} implies that \eqref{eqn: pointwise limit}
holds $\mathbb{P}_0$, and therefore $\mathbb{P}_1$, almost everywhere for all $w \in W$. Fix any $\mu$
where this holds; our goal is to show that there is only one $\nu$ where equality holds in \eqref{eqn:
dual inequality}. To this end, suppose equality holds for some fixed $\nu$. For a fixed $w \in W$,
letting $\mu_t=\Psi^{w,t}(\mu)$, consider the function
\begin{equation*}
\alpha(t) =U(\mu_t) - W_2^2(\mu_t,\nu). 
\end{equation*}
Since $\alpha(t) \leq -V(\nu)$ for all $t$ by \eqref{eqn: dual inequality}, and $\alpha(0) =V(\nu)$ by
assumption (as $\mu_0=\mu)$, $\alpha$ is maximized at $t=0$. Corollary \ref{cor: pointwise
differentiability} implies the differentiability of $t\mapsto U(\mu_t)$ at $t=0$ and Theorem
\ref{Wasserstein-derivative} implies that $t \mapsto W_2^2(\mu_t,\nu)$ is differentiable at $t=0$. It
follows that $\alpha$ is differentiable at $0$, and
\begin{equation}\label{derivatives-equality-result}
\int_M\<\bD U(\mu)(x)-2\nabla\phi_{\nu}(x),w(x)\>_g\,d\mu(x)=\alpha'(0) =0,
\end{equation}
where $\bD U$ is as in Theorem \ref{Schiavo-der} and $\phi_{\nu}$ is the Kantorovich potential
corresponding to the optimal transport between $\mu$ and $\nu$. Since this holds for all $w \in W$, it
holds for all $w \in \Xz$ by density, which implies that
\begin{equation*}
\nabla\phi_\nu(x) =\frac{1}{2}\bD U(\mu)(x)
\end{equation*}
for $\mu$ almost every $x$. But Theorem \ref{thm: McCann Monge solutions on manifolds} then implies that
$\nu$ is uniquely determined by $\mu$ as desired; more precisely, we must have
\begin{equation*}
\nu=\exp\(-\frac{1}{2}\bD U(\mu)\)_{\#}\mu=:T(\mu).
\end{equation*}
This shows that $\mathbb{P}_{1,2}$ concentrates on the graph of the transport map $T$. Uniqueness of the
optimal $\mathbb{P}_{1,2}$ in \eqref{eqn: kantorovich problem} then follows by a very standard argument.
If $\mathbb{P}_{1,2}$ and $\mathbb{P}_{1,2}'$ are both optimal, then by linearity so is
$\frac{1}{2}[\mathbb{P}_{1,2} + \mathbb{P}_{1,2}']$.  The argument above implies that $\mathbb{P}_{1,2},
\mathbb{P}_{1,2}'$ and $\frac{1}{2}[\mathbb{P}_{1,2} + \mathbb{P}_{1,2}']$ are all concentrated on maps
over $\mathbb{P}_1$, which is impossible unless $\mathbb{P}_{1,2}= \mathbb{P}_{1,2}'$.
\endofprf%
\end{prf}
\begin{remark}
The assumption $\RP_1(\pac)=1$ may be relaxed to require only that $\RP_1$ almost every $\mu$ does not
charge $n-1$ dimensional rectifiable sets, or, even further, that $\RP_1$ almost every $\mu$ satisfies
the sharp condition for the existence and uniqueness of optimal maps identified in
\cite{gigli2011inverse}, since it is well known that the conclusion of Theorem \ref{thm: McCann Monge
solutions on manifolds} holds under either of these hypotheses.  Thus the conclusion of Theorem
\ref{Monge-extension} holds when $\RP_0$ is, for example, the entropic measure on $\mathcal{P}(\mathbb
S^1)$ studied in \cite{VonRenesseSturm2009}, which is concentrated on non-atomic singular measures.
\end{remark}
\section{Extension to more general functions of the Wasserstein distance}
The result of Theorem~\ref{Monge-extension} can be further extended to a class of cost functions, namely,
the class of all functions of the form $c(\mu,\nu)=h(W_2(\mu,\nu))$ where $h:
\mathbb{R}_+ \rightarrow
\mathbb{R}$ is $C^2$-smooth, strictly increasing and strictly convex, as the following theorem asserts.
\begin{theorem}
Let $\RP_1,\RP_2$ be Borel probability measures on $X=Y=\PM$, where $\PM$ is equipped with a reference
measure $\RP_0$ satisfying the Rademacher property, such that  $\RP_1$ is absolutely continuous with
respect to $\RP_0$ and $\RP_1(\pac)=1$.  Let the cost function be defined by $c(\mu,\nu)=h(W_2(\mu,\nu))$
for every $\mu,\nu\in\pM$ where $h: \mathbb{R}_+\rightarrow \mathbb{R}$ is $C^2$-smooth, strictly
increasing and strictly convex. Then the solution to \eqref{eqn: kantorovich problem} is unique and
concentrated on the graph of a map $T:X\to Y$ such that $T_\#\RP_1=\RP_2$, and $T$ is the unique solution
to \eqref{eqn: Monge problem}.
\end{theorem}
\begin{prf}
Since much of the proof is very similar to the proof of Theorem \ref{Monge-extension}, we only outline
the differences. First note that we can write
\begin{equation*}
c(\mu,\nu)=h(W_2(\mu,\nu))=\bar{h}(W_2^2(\mu,\nu)).
\end{equation*}
for $\bar{h}(s)=h(\sqrt{s})$. Then, for a curve $\mu_t$ of measures as in Theorem
\ref{Wasserstein-derivative}, \eqref{Wasserstein-der} combined with the chain rule implies
\begin{equation}
\dfrac{d}{dt}\Big|_{t=0} c(\mu_t,\nu) = 2\bar{h}'(W_2^2(\mu,\nu))\int_M\<\nabla\phi(x),w(x)\>_g\,d\mu(x).
\end{equation}
Letting $\RP$ solve \eqref{eqn: kantorovich problem} and $U$, $V$ be $W_2^2$-concave solutions to
\eqref{eqn: dual problem}, a similar argument to the proof of Theorem~\ref{Monge-extension}, implies that
for $\RP_1$ almost every $\mu$ and $\nu$ we have
\begin{align}
\bD U (\mu)(x) = 2\bar{h}'(W_2^2(\mu,\nu))\nabla\phi_\nu(x)\qquad\mu\text{-a.e.}\, 
\end{align}
What is left to show is the fact that for fixed $\mu$ the above equation uniquely determines $\nu$; that
is
\begin{eqnarray}\label{derivative-equality-h}
\bar{h}'(W_2^2(\mu,\nu_1))\nabla\phi_{\nu_1}=\bar{h}'(W_2^2(\mu,\nu_2))\nabla\phi_{\nu_2}
\implies\nabla \phi_{\nu_1}=\nabla \phi_{\nu_2},
\end{eqnarray}
where $\phi_{\nu_i}$ is the Kantorovich potential for transport between $\mu$ and $\nu_i$.
To prove that the above implication holds, we first take the squared norm of both sides of 
(\ref{derivative-equality-h}) and integrate with respect to $\mu$ to obtain
\begin{equation}\label{integral-of-square}
[\bar{h}'(W_2^2(\mu,\nu_1))]^2\int_M\|\nabla\phi_{\nu_1}|^2\,d\mu =
[\bar{h}'(W_2^2(\mu,\nu_2))]^2\int_M\|\nabla\phi_{\nu_2}|^2\,d\mu.
\end{equation}
Observe that $d_g^2(x,\exp_x(-\nabla\phi(x))=\|\nabla\phi(x)|^2$, and therefore
$W_2^2(\mu,\nu)=\int_M\|\nabla\phi|^2\,d\mu$.  This means that (\ref{integral-of-square}) implies
\begin{equation}\label{integral-of-square-equivalent}
\bar{h}'(W_2^2(\mu,\nu_1))W_2(\mu,\nu_1)=\bar{h}'(W_2^2(\mu,\nu_2))W_2(\mu,\nu_2).
\end{equation}
Now note that 
\begin{equation}\label{eqn: h derivatives}
h(s)=\bar{h}(s^2)\implies h'(s)=2s\bar{h}'(s^2),
\end{equation}
and since $h$ is assumed to be strictly convex, the mapping $s \mapsto h'(s)=2s\bar{h}'(s^2)$ is strictly
increasing. So \eqref{integral-of-square-equivalent} implies that $W_2(\mu,\nu_1)=W_2(\mu,\nu_2)$.  If
$\bar h'(W_2^2(\mu,\nu_1)) \neq 0$, we can then cancel this common factor in
\eqref{derivative-equality-h} to obtain the desired result.  On the other hand, if
$\bar{h}'(W_2^2(\mu,\nu_1)) =0$, we must have $h'(W_2(\mu,\nu_1)) = 0$ by \eqref{eqn: h derivatives}.  As
$h$ is strictly increasing and strictly convex, we can only have $h'(s) =0$ at $s=0$, so that
$0=W_2(\mu,\nu_1)=W_2(\mu,\nu_2)$, which of course implies $\nu_1=\nu_2$ $(=\mu)$ in this case as well.
\end{prf}
\section{Declarations}
\begin{itemize}

\item \textbf{Conflict of interest:} The authors have no relevant financial or non-financial interests to disclose.
\item \textbf{Ethics approval:} No ethics approval is required for this article. 
\item \textbf{Funding:} Brendan Pass is pleased to acknowledge support from Natural Sciences and Engineering
Research Council of Canada Grants 04658-2018 and 2024-04864.
\item \textbf{Data availability: }This article has no associated data.
\end{itemize}
\bibliography{refs}

\providecommand{\bysame}{\leavevmode\hbox to3em{\hrulefill}\thinspace}
\providecommand{\MR}{\relax\ifhmode\unskip\space\fi MR }
\providecommand{\MRhref}[2]{%
  \href{http://www.ams.org/mathscinet-getitem?mr=#1}{#2}
}
\providecommand{\href}[2]{#2}
\begin{thebibliography}{10}

\bibitem{ambrosio2008gradient}
Luigi Ambrosio, Nicola Gigli, and Giuseppe Savar{\'e}, \emph{Gradient flows: in
  metric spaces and in the space of probability measures}, Springer Science \&
  Business Media, 2008.

\bibitem{AmbrosioRigot04}
Luigi Ambrosio and S\'{e}verine Rigot, \emph{Optimal mass transportation in the
  {H}eisenberg group}, J. Funct. Anal. \textbf{208} (2004), no.~2, 261--301.
  \MR{2035027}

\bibitem{Bertrand08}
J\'{e}r\^{o}me Bertrand, \emph{Existence and uniqueness of optimal maps on
  {A}lexandrov spaces}, Adv. Math. \textbf{219} (2008), no.~3, 838--851.
  \MR{2442054}

\bibitem{BlanchetMurthy2019}
Jose Blanchet and Karthyek Murthy, \emph{Quantifying distributional model risk
  via optimal transport}, Math. Oper. Res. \textbf{44} (2019), no.~2, 565--600.
  \MR{3959085}

\bibitem{brenier1991polar}
Yann Brenier, \emph{Polar factorization and monotone rearrangement of
  vector-valued functions}, Comm. Pure Appl. Math. \textbf{44} (1991), no.~4,
  375--417.

\bibitem{decreusefond2008wasserstein}
Laurent Decreusefond, \emph{Wasserstein distance on configuration space},
  Potential Anal. \textbf{28} (2008), no.~3, 283--300.

\bibitem{DelloSchiavo2022}
Lorenzo Dello~Schiavo, \emph{The {D}irichlet-{F}erguson diffusion on the space
  of probability measures over a closed {R}iemannian manifold}, Ann. Probab.
  \textbf{50} (2022), no.~2, 591--648. \MR{4399159}

\bibitem{FeyelUstunel04}
Denis Feyel and Ali~Suleyman \"{U}st\"{u}nel, \emph{Monge-{K}antorovitch
  measure transportation and {M}onge-{A}mp\`ere equation on {W}iener space},
  Probab. Theory Related Fields \textbf{128} (2004), no.~3, 347--385.
  \MR{2036490}

\bibitem{FigalliRifford10}
Alessio Figalli and Ludovic Rifford, \emph{Mass transportation on
  sub-{R}iemannian manifolds}, Geom. Funct. Anal. \textbf{20} (2010), no.~1,
  124--159. \MR{2647137}

\bibitem{gangbo2011differential}
Wilfrid Gangbo, Hwa Kim, and Tommaso Pacini, \emph{{Differential forms on
  Wasserstein space and infinite-dimensional Hamiltonian systems}}, vol. 211,
  American Mathematical Society, 2011.

\bibitem{GhossoubSaundersZhang2024}
Mario Ghossoub, David Saunders, and Kelvin~Shuangjian Zhang, \emph{Bounds on
  {C}hoquet risk measures in finite product spaces with ambiguous marginals},
  Stat. Risk Model. \textbf{41} (2024), no.~1-2, 49--72. \MR{4685124}

\bibitem{gigli2011inverse}
Nicola Gigli, \emph{{On the inverse implication of Brenier-McCann theorems and
  the structure of {$\scriptsize{(P_2(M),W_2)}$}}}, Methods Appl. Anal.
  \textbf{18} (2011), no.~2, 127--158.

\bibitem{gigli2009second}
Nicola Gigli, \emph{Second order analysis on {$\scriptsize{(P_2(M),W_2)}$}},
  Mem. Amer. Math. Soc. \textbf{216} (2012), no.~1018, xii+154. \MR{2920736}

\bibitem{kantorovich1942}
Leonid~V Kantorovich, \emph{On the translocation of masses}, Dokl. Akad. Nauk.
  USSR (NS), vol.~37, 1942, pp.~199--201.

\bibitem{Kolesnikov2004}
Alexander~V. Kolesnikov, \emph{Convexity inequalities and optimal transport of
  infinite-dimensional measures}, J. Math. Pures Appl. (9) \textbf{83} (2004),
  no.~11, 1373--1404. \MR{2096305}

\bibitem{lott2008some}
John Lott, \emph{Some geometric calculations on wasserstein space.}, Comm.
  Math. Phys. \textbf{277} (2008), no.~2.

\bibitem{mccann2001polar}
Robert~J McCann, \emph{Polar factorization of maps on riemannian manifolds},
  Geom. Funct. Anal. \textbf{11} (2001), no.~3, 589--608.

\bibitem{monge1781}
Gaspard Monge, \emph{M{\'e}moire sur la th{\'e}orie des d{\'e}blais et des
  remblais}, Mem. Math. Phys. Acad. Royale Sci. (1781), 666--704.

\bibitem{santambrogio2015optimal}
Filippo Santambrogio, \emph{Optimal transport for applied mathematicians},
  vol.~55, Birk{\"a}user, NY, 2015.

\bibitem{schiavo2020rademacher}
Lorenzo~Dello Schiavo, \emph{{A Rademacher-type theorem on L2-Wasserstein
  spaces over closed Riemannian manifolds}}, J. Funct. Anal. \textbf{278}
  (2020), no.~6, 108397.

\bibitem{TrigilaTabak2016}
Giulio Trigila and Esteban~G. Tabak, \emph{Data-driven optimal transport},
  Comm. Pure Appl. Math. \textbf{69} (2016), no.~4, 613--648. \MR{3465084}

\bibitem{villani2021topics}
C{\'e}dric Villani, \emph{Topics in optimal transportation}, vol.~58, American
  Mathematical Soc., 2003.

\bibitem{villani2009}
C\'{e}dric Villani, \emph{Optimal transport: old and new}, Grundlehren der
  mathematischen Wissenschaften [Fundamental Principles of Mathematical
  Sciences], vol. 338, Springer-Verlag, Berlin, 2009, Old and new. \MR{2459454}

\bibitem{VonRenesseSturm2009}
Max-K. von Renesse and Karl-Theodor Sturm, \emph{Entropic measure and
  {W}asserstein diffusion}, Ann. Probab. \textbf{37} (2009), no.~3, 1114--1191.
  \MR{2537551}

\end{thebibliography}
\end{document}